\documentclass[12pt]{article} 
\usepackage{amsfonts,amsmath,latexsym,amssymb,mathrsfs,amsthm,comment}
\usepackage{slashbox}
\usepackage{caption}

\let\OLDthebibliography\thebibliography
\renewcommand\thebibliography[1]{
  \OLDthebibliography{#1}
  \setlength{\parskip}{0pt}
  \setlength{\itemsep}{0pt plus 0.3ex}
}

\def\numberlikeadb{\global\def\theequation{\thesection.\arabic{equation}}}
\numberlikeadb
\newtheorem{theorem}{Theorem}[section]

\newtheorem{corollary}[theorem]{Corollary}

\newtheorem{remark}[theorem]{Remark}

\usepackage{lscape}
\usepackage{caption}
\usepackage{multirow}
\begin{document}

\title{Inequalities for some integrals involving modified Lommel functions of the first kind}
\author{Robert E. Gaunt\footnote{School of Mathematics, The University of Manchester, Manchester M13 9PL, UK}}

\date{\today} 
\maketitle

\vspace{-5mm}

\begin{abstract}In this paper, we obtain inequalities for some integrals involving the modified Lommel function of the first kind $t_{\mu,\nu}(x)$.  In most cases, these inequalities are tight in certain limits.  We also deduce a tight double inequality, involving the modified Lommel function $t_{\mu,\nu}(x)$, for a generalized hypergeometric function.  The inequalities obtained in this paper generalise recent bounds for integrals involving the modified Struve function of the first kind.
\end{abstract}


\noindent{{\bf{Keywords:}}} Modified Lommel function; inequality; integral

\noindent{{{\bf{AMS 2010 Subject Classification:}}} Primary 33C20; 26D15

\section{Introduction}\label{intro}

In a series of recent papers \cite{gaunt ineq1, gaunt ineq3, gaunt ineq6, gaunt ineq4, gaunt ineq42}, simple lower and upper bounds, involving the modified Bessel function of the first kind $I_\nu(x)$ and the modified Struve function of the first kind $\mathbf{L}_\nu(x)$, respectively, were obtained for the integrals 
\begin{equation}\label{intbes}\int_0^x \mathrm{e}^{-\beta u} u^{\pm\nu} I_\nu(u)\,\mathrm{d}u, \qquad \int_0^x \mathrm{e}^{-\beta u} u^{\pm\nu} \mathbf{L}_\nu(u)\,\mathrm{d}u,
\end{equation}
where $x>0$, $0\leq\beta<1$.  The conditions imposed on $\nu$ were different for several of the inequalities.  Inequalities for some other closely related integrals were also obtained.  For $\beta\not=0$ there does not exist simple closed-form expressions for the integrals in (\ref{intbes}). The inequalities of \cite{gaunt ineq1,gaunt ineq3} were essential in the development of Stein's method \cite{stein,chen,np12} for variance-gamma approximation \cite{eichelsbacher, gaunt vg, gaunt vg2}.  Moreover, as the inequalities are simple and surprisingly accurate they may also prove useful in other problems involving modified Bessel functions; see for example, \cite{bs09,baricz3} in which inequalities for modified Bessel functions are used to obtain tight bounds for the generalized Marcum Q-function, which frequently arises in radar signal processing.

In this paper, we address the natural problem of obtaining simple inequalities, involving the modified Lommel function $t_{\mu,\nu}(x)$, for the integrals
\begin{equation}\label{intstruve}\int_0^x \mathrm{e}^{-\beta u} u^{\pm\nu}  t_{\mu,\nu}(u)\,\mathrm{d}u,
\end{equation}
where $x>0$, $0\leq\beta<1$ and the conditions on $\mu$ and $\nu$ will vary from inequality to inequality.  We will also establish bounds for some closely related integrals.  Up to a multiplicative constant, the modified Lommel function $t_{\mu,\nu}(x)$ generalises the modified Struve function $\mathbf{L}_\nu(x)$ (see (\ref{bnm})), and a number of the properties of $\mathbf{L}_\nu(x)$ that were exploited in derivations of the inequalities for the integrals in (\ref{intbes}) by \cite{gaunt ineq4, gaunt ineq42} generalise in a natural manner.  As such, the bounds obtained in this paper generalise those of \cite{gaunt ineq4, gaunt ineq42}.

Modified Lommel functions are widely used special functions, arising in areas of the applied sciences as diverse as the theory of steady-state temperature distribution \cite{g29}, scattering amplitudes in quantum optics \cite{t73} and stress distributions in cylindrical objects \cite{s85}; see \cite{gaunt lommel} for a list of further applications. The modified Lommel function of the first kind $t_{\mu,\nu}(x)$ is defined by the hypergeometric series
\begin{align}\label{hypt}t_{\mu,\nu}(x)&=\frac{x^{\mu+1}}{(\mu-\nu+1)(\mu+\nu+1)} {}_1F_2\bigg(1;\frac{\mu-\nu+3}{2},\frac{\mu+\nu+3}{2};\frac{x^2}{4}\bigg) \\
&=2^{\mu-1}\Gamma\big(\tfrac{\mu-\nu+1}{2}\big)\Gamma\big(\tfrac{\mu+\nu+1}{2}\big)\sum_{k=0}^\infty\frac{(\frac{1}{2}x)^{\mu+2k+1}}{\Gamma\big(k+\frac{\mu-\nu+3}{2}\big)\Gamma\big(k+\frac{\mu+\nu+3}{2}\big)},\nonumber
\end{align}
and arises as a particular solution of the modified Lommel differential equation \cite{s36,r64}
\begin{equation*}x^2f''(x)+xf'(x)-(x^2+\nu^2)f(x)=x^{\mu+1}.
\end{equation*}
In the literature different notation is used for the modified Lommel functions; we adopt that of \cite{zs13}.  The terminology modified Lommel function of the \emph{first kind} is also not standard in the literature, but has recently been adopted by \cite{gaunt lommel}.  Also, \cite{bk14} have used the terminology Lommel function of the \emph{first kind} for the function $s_{\mu,\nu}(x)$, which is related to the modified Lommel function of the first kind by $t_{\mu,\nu}(x)=-\mathrm{i}^{1-\mu}s_{\mu,\nu}(\mathrm{i}x)$ (see \cite{r64,zs13}). From this relationship many properties of modified Lommel functions can be inferred from those of Lommel functions that are given in standard references, such as \cite{b67,l69,olver,g44}.

For the purposes of this paper, we follow \cite{gaunt lommel} and use the following normalization which will remove a number of multiplicative constants from our calculations:
\begin{align}\label{tseries}\tilde{t}_{\mu,\nu}(x)&=\frac{1}{2^{\mu-1}\Gamma\big(\frac{\mu-\nu+1}{2}\big)\Gamma\big(\frac{\mu+\nu+1}{2}\big)}t_{\mu,\nu}(x)=\sum_{k=0}^\infty\frac{(\frac{1}{2}x)^{\mu+2k+1}}{\Gamma\big(k+\frac{\mu-\nu+3}{2}\big)\Gamma\big(k+\frac{\mu+\nu+3}{2}\big)}.
\end{align}
For ease of exposition, we shall also refer to $\tilde{t}_{\mu,\nu}(x)$ as the modified Lommel function of the first kind.  From this point on, we shall work with the function $\tilde{t}_{\mu,\nu}(x)$.  Results for $t_{\mu,\nu}(x)$ can be easily inferred.  As an example, which is relevant to the inequalities obtained in this paper,  using the formula $u\Gamma(u)=\Gamma(u+1)$ gives
\begin{align*}\frac{t_{\mu,\nu}(x)}{\tilde{t}_{\mu,\nu}(x)}=(\mu+\nu-1)\frac{t_{\mu-1,\nu-1}(x)}{\tilde{t}_{\mu-1,\nu-1}(x)}.
\end{align*}
We record the important special case
\begin{equation}\label{bnm}\tilde{t}_{\nu,\nu}(x)= \mathbf{L}_\nu(x).
\end{equation}



When $\beta=0$ the integrals in (\ref{intstruve}) can be evaluated exactly in terms of the generalized hypergeometric function.   A straightforward calculation involving simple manipulations of the formula (\ref{hypt}) (followed by the normalization in (\ref{tseries})), yields
\begin{align}\int_0^x u^\alpha \tilde{t}_{\mu,\nu}(u)\,\mathrm{d}u&=\frac{x^{\mu+\alpha+2}}{2^{\mu+1}(\mu+\alpha+2)\Gamma\big(\frac{\nu-\nu+3}{2}\big)\Gamma\big(\frac{\mu+\nu+3}{2}\big)}\nonumber\\
\label{besint6}&\quad\times{}_2F_3\bigg(1,\frac{\mu+\alpha+2}{2};\frac{\mu-\nu+3}{2},\frac{\mu+\nu+3}{2},\frac{\mu+\alpha+4}{2};\frac{x^2}{4}\bigg), 
\end{align}
where we require $\mu+\alpha>-2$ for the integral to exist.  When $\beta\not=0$ a more complicated formula is available:
\begin{equation*}\int_0^x \mathrm{e}^{-\beta u} u^{\pm\nu}  \tilde{t}_{\mu,\nu}(u)\,\mathrm{d}u=\sum_{k=0}^\infty\frac{2^{-\mu-2k-1}\beta^{-2k-\mu\mp\nu-2}}{\Gamma\big(k+\frac{\mu-\nu+3}{2}\big)\Gamma\big(k+\frac{\mu+\nu+3}{2}\big)}\gamma(\mu\pm\nu+2k+2,\beta x),
\end{equation*}
where  $\gamma(a,x)=\int_0^x u^{a-1}\mathrm{e}^{-u}\,\mathrm{d}u$ is the lower incomplete gamma function, and we require $\mu\pm\nu>-2$ for the integral to exist.  These complicated formulas provide the motivation for establishing simple bounds, involving the modified Lommel function $ \tilde{t}_{\mu,\nu}(x)$ itself, for the integrals in (\ref{intstruve}).


The approach taken in this paper to bound the integrals in (\ref{intstruve}) is similar to that used by \cite{gaunt ineq4,gaunt ineq42} to bound the corresponding integrals involving the modfied Struve function $\mathbf{L}_\nu(x)$, and the inequalities obtained in this paper generalise those of \cite{gaunt ineq4,gaunt ineq42} in a natural manner.  In spite of their simple form, in most cases, the bounds obtained in this paper will be tight in certain limits. 

As already noted, the properties of the modified Struve function $\mathbf{L}_\nu(x)$ that were exploited in the proofs of  \cite{gaunt ineq4,gaunt ineq42} are shared by the modified Lommel function $ \tilde{t}_{\mu,\nu}(x)$, which we now list.  
With the exception of the differentiation formula (\ref{diffone}) (see \cite{r64}),  the following basic properties can all be found in \cite{gaunt lommel}.  For $x>0$, the function $ \tilde{t}_{\mu,\nu}(x)$ is positive if $\mu-\nu\geq-3$ and $\mu+\nu\geq-3$. The function $ \tilde{t}_{\mu,\nu}(x)$ satisfies the recurrence relations and differentiation formula 
\begin{align}\label{Iidentity}\tilde{t}_{\mu-1,\nu-1}(x)-\tilde{t}_{\mu+1,\nu+1}(x)&=\frac{2\nu}{x}\tilde{t}_{\mu,\nu}(x)+a_{\mu,\nu}(x), \\
\label{lomrel2}\tilde{t}_{\mu-1,\nu-1}(x)+\tilde{t}_{\mu+1,\nu+1}(x)&=2\tilde{t}_{\mu,\nu}'(x)-a_{\mu,\nu}(x), \\
\label{diffone}\frac{\mathrm{d}}{\mathrm{d}x} \big(x^{\nu}  \tilde{t}_{\mu,\nu} (x) \big) &= x^{\nu} \tilde{t}_{\mu-1,\nu -1} (x),
\end{align}
where $a_{\mu,\nu}(x)=\frac{(x/2)^\mu}{\Gamma(\frac{\mu-\nu+1}{2})\Gamma(\frac{\mu+\nu+3}{2})}$.  We shall also need another differentation formula that is not given in \cite{gaunt lommel} or \cite{r64}. With the aid of (\ref{Iidentity}) and (\ref{lomrel2}) we obtain
\begin{align}\frac{\mathrm{d}}{\mathrm{d}x}\bigg(\frac{\tilde{t}_{\mu,\nu} (x)}{x^\nu}\bigg)&=-\frac{\nu}{x^{\nu+1}}\tilde{t}_{\mu,\nu} (x)+\frac{1}{x^\nu}\tilde{t}_{\mu,\nu}' (x)\nonumber \\
&=-\frac{1}{x^\nu}\big(\tilde{t}_{\mu-1,\nu-1}(x)-\tilde{t}_{\mu+1,\nu+1}(x)-a_{\mu,\nu}(x)\big)\nonumber\\
&\quad+\frac{1}{x^\nu}\big(\tilde{t}_{\mu-1,\nu-1}(x)+\tilde{t}_{\mu+1,\nu+1}(x)+a_{\mu,\nu}(x)\big) \nonumber\\
\label{diffone1}&=\frac{\tilde{t}_{\mu+1,\nu+1} (x)}{x^\nu}+\frac{a_{\mu,\nu}(x)}{x^\nu}.
\end{align}
The function $ \tilde{t}_{\mu,\nu}(x)$ has the following asymptotic properties  \cite{gaunt lommel}:
\begin{align}\label{Itend0} \tilde{t}_{\mu,\nu}(x)&\sim\frac{(\frac{1}{2}x)^{\mu+1}}{\Gamma\big(\frac{\mu-\nu+3}{2}\big)\Gamma\big(\frac{\mu+\nu+3}{2}\big)}\bigg(1+\frac{x^2}{(\mu+3)^2-\nu^2}\bigg), \quad x\downarrow0,\:\mu>-3,\:|\nu|<\mu+3, \\
\label{Itendinfinity} \tilde{t}_{\mu,\nu}(x)&\sim\frac{\mathrm{e}^x}{\sqrt{2\pi x}}\bigg(1-\frac{4\nu^2-1}{8x}+\frac{(4\nu^2-1)(4\nu^2-9)}{128x^2}\bigg), \quad x\rightarrow\infty, \:\mu,\nu\in\mathbb{R}.
\end{align}
Let $x > 0$, $\mu>-\frac{1}{2}$ and $\frac{1}{2}\leq\nu<\mu+1$. Then 
\begin{equation}\label{Imon} \tilde{t}_{\mu,\nu} (x) < \tilde{t}_{\mu-1,\nu - 1} (x).
\end{equation} 
This inequality was obtained by \cite{gaunt lommel}, and generalises an inequality of \cite{bp14} for the modified Struve function $\mathbf{L}_\nu(x)$.  For other functional inequalities and monotonicity results involving modified Lommel functions of the first kind see \cite{m17}.

\section{Inequalities for integrals of modified Lommel functions}\label{sec2}

The inequalities in the following two theorems for integrals of the type $\int_0^x\mathrm{e}^{-\beta u}u^\nu\tilde{t}_{\mu,\nu}(u)\,\mathrm{d}u$ and $\int_0^x\mathrm{e}^{-\beta u}u^{-\nu}\tilde{t}_{\mu,\nu}(u)\,\mathrm{d}u$ generalise the inequalities of Theorem 2.1 of \cite{gaunt ineq4} and Theorem 2.1 of \cite{gaunt ineq42} for analogous integrals involving the modified Struve function of the first kind.   
Before stating the theorems, we introduce the notation
\begin{align*}a_{\mu,\nu,n}&=\frac{n+1}{2^{\mu+n+1}(2\nu+n+1)(\mu+\nu+n+2)\Gamma\big(\frac{\mu-\nu+1}{2}\big)\Gamma\big(\frac{\mu+\nu+2n+5}{2}\big)}, \\
b_{\mu,\nu,n}&=\frac{2\nu+n+1}{2^{\mu+n+2}(\mu-\nu+n+2)(\nu+n+1)\Gamma\big(\frac{\mu-\nu+1}{2}\big)\Gamma\big(\frac{\mu+\nu+2n+5}{2}\big)}, \\
c_{\mu,\nu,n}&=\frac{(2\nu+n+1)(2\nu+n+3)}{2^{\mu+n+4}(n+1)(\mu-\nu+n+4)(\nu+n+3)\Gamma\big(\frac{\mu-\nu+1}{2}\big)\Gamma\big(\frac{\mu+\nu+2n+9}{2}\big)}, \\
d_{\mu,\nu,n}&=\frac{2\nu+n+1}{2^{\mu+n+1}(n+1)(\mu-\nu+n+2)\Gamma\big(\frac{\mu-\nu+1}{2}\big)\Gamma\big(\frac{\mu+\nu+2n+5}{2}\big)}.
\end{align*}

\begin{theorem} \label{tiger} Let $n>-1$ and $0\leq \beta <1$. Then, for all $x>0$,
\begin{align}\label{besi11}\int_0^x \mathrm{e}^{-\beta u}u^\nu  \tilde{t}_{\mu+n,\nu+n}(u)\,\mathrm{d}u&>\mathrm{e}^{-\beta x}x^\nu  \tilde{t}_{\mu+n+1,\nu+n+1}(x),   \\
&\quad\quad\quad\quad\mu>-\tfrac{1}{2}(n+5),\:-n-\mu-2<\nu\leq\mu+3, \nonumber \\
\label{100fcp}\int_0^xu^{\nu} \tilde{t}_{\mu,\nu}(u)\,\mathrm{d}u &<x^{\nu} \tilde{t}_{\mu,\nu}(x), \quad \mu>-\tfrac{1}{2},\:\tfrac{1}{2}\leq\nu<\mu+1, \\
\label{besi22}\int_0^x u^\nu  \tilde{t}_{\mu+n,\nu+n}(u)\,\mathrm{d}u&<\frac{x^\nu}{2\nu+n+1}\bigg(2(\nu+n+1) \tilde{t}_{\mu+n+1,\nu+n+1}(x)\nonumber \\
& \quad -(n+1) \tilde{t}_{\mu+n+3,\nu+n+3}(x)\bigg)-a_{\mu,\nu,n}x^{\mu+\nu+n+2},\nonumber \\
&\quad\quad\quad\quad \mu>-\tfrac{1}{2}(n+3), \: -\tfrac{1}{2}(n+1)<\nu<\mu+1,\\
\label{pron}\int_0^x\mathrm{e}^{-\beta u} u^{\nu} \tilde{t}_{\mu,\nu}(u)\,\mathrm{d}u &\leq \frac{\mathrm{e}^{-\beta x}}{1-\beta}\int_0^xu^{\nu} \tilde{t}_{\mu,\nu}(u)\,\mathrm{d}u, \quad \mu>-\tfrac{1}{2},\:\tfrac{1}{2}\leq\nu<\mu+1, \\
\label{besi33}\int_0^x \mathrm{e}^{-\beta u}u^{\nu} \tilde{t}_{\mu,\nu}(u)\,\mathrm{d}u&<\frac{\mathrm{e}^{-\beta x}x^\nu}{(2\nu+1)(1-\beta)}\bigg(2(\nu+1) \tilde{t}_{\mu+1,\nu+1}(x)- \tilde{t}_{\mu+3,\nu+3}(x)\bigg)\nonumber \\
&\quad-a_{\mu,\nu,n}x^{\mu+\nu+n+2}, \quad \mu>-\tfrac{1}{2},\:\tfrac{1}{2}\leq\nu<\mu+1, \\
\label{besi44}\int_0^x \mathrm{e}^{-\beta u}u^{\nu+1}  \tilde{t}_{\mu,\nu}(u)\,\mathrm{d}u&\geq\mathrm{e}^{-\beta x}x^{\nu+1}  \tilde{t}_{\mu+1,\nu+1}(x), \quad \mu>-3,\:|\nu|<\mu+3, \\
\label{besi55}\int_0^x \mathrm{e}^{-\beta u}u^{\nu+1}  \tilde{t}_{\mu,\nu}(u)\,\mathrm{d}u&\leq\frac{1}{1-\beta}\mathrm{e}^{-\beta x}x^{\nu+1}  \tilde{t}_{\mu+1,\nu+1}(x), \quad\mu>-\tfrac{3}{2},\:-\tfrac{1}{2}\leq\nu<\mu+1.
\end{align}
We have equality in (\ref{pron}), (\ref{besi44}) and (\ref{besi55}) if and only if $\beta=0$.  Inequalities (\ref{100fcp})--(\ref{besi55}) are tight as $x\rightarrow\infty$, and inequalities (\ref{besi22}) and (\ref{besi44}) are also tight as $x\downarrow0$. Inequality (\ref{besi11}) is tight as $x\rightarrow\infty$ if $\beta=0$.  Inequalities (\ref{besi11}) and (\ref{besi44}) hold for all $\beta\geq0$.
\end{theorem}


\begin{proof}We prove inequalities (\ref{besi11})--(\ref{besi55}), before verifying that they are tight in certain limits.

(i)  Let us first prove inequality (\ref{besi11}).  The conditions on $\mu$, $\nu$ and $n$ imply that $\mu+\nu+n>-2$, and so the integral exists.  The conditions also imply that $\mu-\nu\geq-3$ and $\mu+\nu+2n\geq-3$, and therefore $\tilde{t}_{\mu+n,\nu+n}(x)>0$ for all $x>0$.  (The conditions on $\mu$, $\nu$ and $n$ for the other inequalities will also always guarantee that the integrals exist and that the modified Lommel functions are positive for all $x>0$, and we will not comment on this further in the proof of these inequalities.) 
Now, since $\beta\geq0$ and $n>-1$, on using the differentiation formula (\ref{diffone}) we have 
\begin{align*}\int_0^x\mathrm{e}^{-\beta u}u^{\nu} \tilde{t}_{\mu+n,\nu+n}(u)\,\mathrm{d}u &=\int_0^x\mathrm{e}^{-\beta u}\frac{1}{u^{n+1}}u^{\nu+n+1} \tilde{t}_{\mu+n,\nu+n}(u)\,\mathrm{d}u\\
& >\frac{\mathrm{e}^{-\beta x}}{x^{n+1}}\int_0^xu^{\nu+n+1} \tilde{t}_{\mu+n,\nu+n}(u)\,\mathrm{d}u =\mathrm{e}^{-\beta x}x^{\nu} \tilde{t}_{\mu+n+1,\nu+n+1}(x),
\end{align*}
as by (\ref{Itend0}) we have $\lim_{x\downarrow 0}x^{\nu+n+1} \tilde{t}_{\mu+n+1,\nu+n+1}(x)=0$ if $\mu+\nu+n>-2$.

(ii) Using inequality (\ref{Imon}) (which is valid for $\mu>-\frac{1}{2}$, $\frac{1}{2}\leq\nu<\mu+1$) and then applying (\ref{diffone}) gives the inequality
\[\int_0^xu^{\nu} \tilde{t}_{\mu,\nu}(u)\,\mathrm{d}u <\int_0^xu^{\nu} \tilde{t}_{\mu-1,\nu-1}(u)\,\mathrm{d}u=x^{\nu} \tilde{t}_{\mu,\nu}(x).\]

(iii) Let us first note that an application of the differentiation formula (\ref{diffone}) and the relation (\ref{Iidentity}) gives that
\begin{align*} &\frac{\mathrm{d}}{\mathrm{d}u} \big(u^{\nu}  \tilde{t}_{\mu+n+1,\nu+n+1} (u)\big)= \frac{\mathrm{d}}{\mathrm{d}u}(u^{-(n+1)} \cdot u^{\nu +n+1}  \tilde{t}_{\mu+n+1,\nu+n+1} (u)) \\
&\quad = u^{\nu}  \tilde{t}_{\mu+n,\nu+n} (u) -(n+1)u^{\nu -1}  \tilde{t}_{\mu+n+1,\nu+n+1}(u) \\
& \quad= u^{\nu}  \tilde{t}_{\mu+n,\nu+n} (u) - \frac{n+1}{2(\nu +n+1)} u^{\nu}  \tilde{t}_{\mu+n,\nu+n} (u) + \frac{n+1}{2(\nu +n+1)} u^{\nu}  \tilde{t}_{\mu+n+2,\nu+n+2} (u) \\
&\quad\quad +(n+1)u^{\nu-1}\cdot\frac{u}{2(\nu+n+1)}a_{\mu+n+1,\nu+n+1}(u) 
\end{align*}
\begin{align*}
&\quad = \frac{2\nu +n+1}{2(\nu +n+1)} u^{\nu}  \tilde{t}_{\mu+n,\nu+n} (u) + \frac{n+1}{2(\nu +n+1)} u^{\nu}  \tilde{t}_{\mu+n+2,\nu+n+2} (u)\\
&\quad\quad+\frac{n+1}{2(\nu+n+1)}u^\nu a_{\mu+n+1,\nu+n+1}(u). 
\end{align*}
Now, on integrating both sides over $(0,x)$, applying the fundamental theorem of calculus and rearranging we obtain
\begin{align*}&\int_0^x u^{\nu}  \tilde{t}_{\mu+n,\nu+n} (u)\,\mathrm{d}u \\
&\quad= \frac{2(\nu +n+1)}{2\nu +n+1} x^{\nu}  \tilde{t}_{\mu+n+1,\nu+n+1} (x) - \frac{n+1}{2\nu +n+1} \int_0^x u^{\nu}  \tilde{t}_{\mu+n+2,\nu+n+2} (u)\,\mathrm{d}u \\
&\quad\quad-\frac{n+1}{2\nu+n+1}\int_0^xu^\nu a_{\mu+n+1,\nu+n+1}(u)\,\mathrm{d}u \\
&\quad= \frac{2(\nu +n+1)}{2\nu +n+1} x^{\nu}  \tilde{t}_{\mu+n+1,\nu+n+1} (x) - \frac{n+1}{2\nu +n+1} \int_0^x u^{\nu}  \tilde{t}_{\mu+n+2,\nu+n+2} (u)\,\mathrm{d}u \\
&\quad\quad-a_{\mu,\nu,n}x^{\mu+\nu+n+2}.
\end{align*}
As the conditions on $\nu$ and $n$ ensure that $\frac{n+1}{2\nu+n+1}>0$, using inequality (\ref{besi11}) with $\beta=0$ to bound the integral on the right hand-side of the above expression then gives inequality (\ref{besi22}).   
 
(iv) Integration by parts and an application of inequality (\ref{100fcp}) gives
\begin{align*} \int_0^x \mathrm{e}^{-\beta u} u^\nu \tilde{t}_{\mu,\nu}(u) \,\mathrm{d}u &= \mathrm{e}^{-\beta x}\int_0^x u^\nu  \tilde{t}_{\mu,\nu}(u)\,\mathrm{d}u + \beta \int_0^x \mathrm{e}^{-\beta u}\bigg(\int_0^u  y^\nu \tilde{t}_{\mu,\nu}(y) \,\mathrm{d}y\bigg) \,\mathrm{d}u \\ &< \mathrm{e}^{-\beta x}\int_0^x u^\nu  \tilde{t}_{\mu,\nu}(u)\,\mathrm{d}u + \beta \int_0^x \mathrm{e}^{-\beta u} u^\nu \tilde{t}_{\mu,\nu}(u)  \,\mathrm{d}u, 
\end{align*}
and on rearranging we have inequality (\ref{pron}).

(v) Combine parts (iii) and (iv).

(vi)  Since $\beta\geq0$, we have
\begin{equation*}\int_0^x \mathrm{e}^{-\beta u}u^{\nu+1} \tilde{t}_{\mu,\nu}(u)\,\mathrm{d}u\geq\mathrm{e}^{-\beta x}\int_0^x u^{\nu+1} \tilde{t}_{\mu,\nu}(u)\,\mathrm{d}u=\mathrm{e}^{-\beta x}x^{\nu+1} \tilde{t}_{\mu+1,\nu+1}(x),
\end{equation*}
with equality if and only if $\beta=0$.

(vii) Let us consider the function
\begin{equation*}v(x)=\frac{1}{1-\beta}\mathrm{e}^{-\beta x}x^{\nu+1} \tilde{t}_{\mu+1,\nu+1}(x)-\int_0^x\mathrm{e}^{-\beta u}u^{\nu+1} \tilde{t}_{\mu,\nu}(u)\,\mathrm{d}u.
\end{equation*}
We prove the result by arguing that $v(x)\geq0$ for all $x>0$.  Using the differentiation formula (\ref{diffone}) followed by inequality (\ref{Imon}) we have that
\begin{align*}v'(x)&=\frac{1}{1-\beta}\mathrm{e}^{-\beta x}x^{\nu+1}\big( \tilde{t}_{\mu,\nu}(x)-\beta  \tilde{t}_{\mu+1,\nu+1}(x)\big)-\mathrm{e}^{-\beta x}x^{\nu+1} \tilde{t}_{\mu,\nu}(x) \\
&=\frac{\beta}{1-\beta}\mathrm{e}^{-\beta x}x^{\nu+1}\big( \tilde{t}_{\mu,\nu}(x)- \tilde{t}_{\mu+1,\nu+1}(x)\big)\geq0,
\end{align*}
and therefore $v$ is a non-decreasing function of $x$ on $(0,\infty)$ Also, from (\ref{Itend0}), as $x\downarrow0$,
\begin{align}v(x)&\sim \frac{1}{1-\beta}\frac{x^{\mu+\nu+3}}{2^{\mu+2}\Gamma\big(\frac{\mu-\nu+3}{2}\big)\Gamma\big(\frac{\mu+\nu+5}{2}\big)}-\int_0^x \frac{u^{\mu+\nu+2}}{2^{\mu+1}\Gamma\big(\frac{\mu-\nu+3}{2}\big)\Gamma\big(\frac{\mu+\nu+3}{2}\big)}\,\mathrm{d}u\nonumber\\
&=\frac{1}{1-\beta}\frac{x^{\mu+\nu+3}}{2^{\mu+2}\Gamma\big(\frac{\mu-\nu+3}{2}\big)\Gamma\big(\frac{\mu+\nu+5}{2}\big)}-\frac{x^{\mu+\nu+3}}{2^{\mu+1}(\mu+\nu+3)\Gamma\big(\frac{\mu-\nu+3}{2}\big)\Gamma\big(\frac{\mu+\nu+3}{2}\big)}\nonumber \\
\label{hujn}& =\frac{\beta}{1-\beta}\frac{x^{\mu+\nu+3}}{2^{\mu+2}\Gamma\big(\frac{\mu-\nu+3}{2}\big)\Gamma\big(\frac{\mu+\nu+5}{2}\big)}\geq0.
\end{align}
Therefore $v(x)\geq0$ for all $x>0$, as required. Clearly, the above argument shows that if $0<\beta<1$ then $v(x)>0$ for all $x>0$. 

(viii) Finally, we establish the tightness of the inequalities as described in the statement of the theorem.
To this end, we note that a straightforward asymptotic analysis using (\ref{Itendinfinity}) gives that, for $0\leq\beta<1$ and $\mu+\nu+n>-2$, 
\begin{equation}\label{intiinf} \int_0^x \mathrm{e}^{-\beta u}u^\nu  \tilde{t}_{\mu+n,\nu+n}(u)\,\mathrm{d}u\sim \frac{1}{\sqrt{2\pi}(1-\beta)}x^{\nu-\frac{1}{2}}\mathrm{e}^{(1-\beta)x}, \quad x\rightarrow\infty,
\end{equation}
and we also have
\begin{equation}\label{sdfg}\mathrm{e}^{-\beta x}x^\nu \tilde{t}_{\mu+n,\nu+n}(x)\sim \frac{1}{\sqrt{2\pi}}x^{\nu-\frac{1}{2}}\mathrm{e}^{(1-\beta)x}, \quad x\rightarrow\infty.
\end{equation}
From (\ref{intiinf}) and (\ref{sdfg}) it is readily seen that inequalities (\ref{100fcp})--(\ref{besi55}) are tight as $x\rightarrow\infty$, and that this is also so for (\ref{besi11}) if $\beta=0$.

Setting in $\beta=0$ in (\ref{hujn}) shows that (\ref{besi44}) is tight as $x\downarrow0$. It now remains to prove that (\ref{besi22}) is tight as $x\downarrow0$.   From (\ref{Itend0}), we have that, as $x\downarrow0$,
\begin{align}\mathrm{LHS}=\int_0^x u^\nu  \tilde{t}_{\mu+n,\nu+n}(u)\,\mathrm{d}u&\sim\int_0^x \frac{u^{\mu+\nu+n+1}}{2^{\mu+n+1}\Gamma\big(\frac{\mu-\nu+3}{2}\big)\Gamma\big(\frac{\mu+\nu+2n+3}{2}\big)}\,\mathrm{d}u\nonumber\\
&=\frac{x^{\mu+\nu+n+2}}{2^{\mu+n+1}(\mu+\nu+n+2)\Gamma\big(\frac{\mu-\nu+3}{2}\big)\Gamma\big(\frac{\mu+\nu+2n+3}{2}\big)},\nonumber
\end{align}
and 
\begin{align}
\mathrm{RHS}&\sim\frac{x^{\nu}}{2\nu+n+1}\cdot\frac{2(\nu+n+1)x^{\mu+n+2}}{2^{\mu+n+2}\Gamma\big(\frac{\mu-\nu+3}{2}\big)\Gamma\big(\frac{\mu+\nu+2n+5}{2}\big)}\nonumber\\
&\quad-\frac{(n+1)x^{\mu+\nu+n+2}}{2^{\mu+n+1}(2\nu+n+2)\Gamma\big(\frac{\mu-\nu+1}{2}\big)\Gamma\big(\frac{\mu+\nu+2n+5}{2}\big)} \nonumber \\
&=\frac{(\mu+\nu+2n+3)x^{\mu+\nu+n+2}}{2^{\nu+n+2}(\mu+\nu+n+2)\Gamma\big(\frac{\mu-\nu+3}{2}\big)\Gamma\big(\frac{\mu+\nu+2n+5}{2}\big)} \nonumber\\
&=\frac{x^{\mu+\nu+n+2}}{2^{\mu+n+1}(\mu+\nu+n+2)\Gamma\big(\frac{\mu-\nu+3}{2}\big)\Gamma\big(\frac{\mu+\nu+2n+3}{2}\big)},\nonumber
\end{align}
as we required.
\end{proof}

\begin{theorem}\label{tiger1}Let $0<\beta<1$ and $n>-1$. Then, for all $x>0$,
\begin{align}\label{bi1}\int_0^x \frac{ \tilde{t}_{\mu,\nu}(u)}{u^\nu}\,\mathrm{d}u&>\frac{ \tilde{t}_{\mu,\nu}(x)}{x^\nu}-\frac{x^{\mu-\nu+1}}{2^{\mu+1}\Gamma\big(\frac{\mu-\nu+3}{2}\big)\Gamma\big(\frac{\mu+\nu+3}{2}\big)}, \\
\label{bi2}\int_0^x \frac{ \tilde{t}_{\mu+n,\nu+n}(u)}{u^\nu}\,\mathrm{d}u&>\frac{ \tilde{t}_{\mu+n+1,\nu+n+1}(x)}{x^\nu}-b_{\mu,\nu,n}x^{n+2},  \\
\label{bi3}\int_0^x \frac{ \tilde{t}_{\mu+n,\nu+n}(u)}{u^\nu}\,\mathrm{d}u&<\frac{2(\nu+n+1)}{n+1}\frac{ \tilde{t}_{\mu+n+1,\nu+n+1}(x)}{x^\nu}-\frac{2\nu+n+1}{n+1}\frac{ \tilde{t}_{\mu+n+3,\nu+n+3}(x)}{x^\nu}\nonumber \\
&\quad+c_{\mu,\nu,n}x^{\mu-\nu+n+4}-d_{\mu,\nu,n}x^{\mu-\nu+n+2},  \\
\label{bi4}\int_0^x \mathrm{e}^{-\beta u}\frac{ \tilde{t}_{\mu,\nu}(u)}{u^\nu}\,\mathrm{d}u&>\frac{1}{1-\beta}\bigg(\mathrm{e}^{-\beta x}\int_0^x\frac{ \tilde{t}_{\mu,\nu}(u)}{u^\nu}\,\mathrm{d}u-\frac{\beta^{-(\mu-\nu+1)}\gamma(\mu-\nu+2,\beta x)}{2^{\mu+1}\Gamma\big(\frac{\mu-\nu+3}{2}\big)\Gamma\big(\frac{\mu+\nu+3}{2}\big)}\bigg), \\
\label{bi5}\int_0^x \mathrm{e}^{-\beta u}\frac{ \tilde{t}_{\mu,\nu}(u)}{u^\nu}\,\mathrm{d}u&>\frac{1}{1-\beta}\bigg(\mathrm{e}^{-\beta x}\frac{ \tilde{t}_{\mu,\nu}(x)}{x^\nu}-\frac{(\beta x)^{\mu-\nu+1}+\gamma(\mu-\nu+2,\beta x)}{2^{\mu+1}\beta^{\mu-\nu+1}\Gamma\big(\frac{\mu-\nu+3}{2}\big)\Gamma\big(\frac{\mu+\nu+3}{2}\big)}\bigg).
\end{align}
Inequalities (\ref{bi1}), (\ref{bi4}) and (\ref{bi5}) hold for $\mu>-\frac{3}{2}$, $-\frac{1}{2}\leq\nu<\mu+1$, and inequalities (\ref{bi2}) and (\ref{bi3}) are valid for $\nu>-\frac{1}{2}(n+3)$, $-\frac{1}{2}(n+1)<\nu<\mu+1$.  We have equality in (\ref{bi2}) and (\ref{bi3}) if $2\nu+n=-1$.  Inequalities (\ref{bi1})--(\ref{bi5}) are tight as $x\rightarrow\infty$ and inequality (\ref{bi3}) is also tight as $x\downarrow0$.  Here $\gamma(a,x)=\int_0^x u^{a-1}\mathrm{e}^{-u}\,\mathrm{d}u$ is the lower incomplete gamma function. 
\end{theorem}

\begin{proof}We restrict out attention to proving the inequalities; proving that the bounds are tight in the limits $x\downarrow0$ and $x\rightarrow\infty$ is similar to that carried out in the proof of Theorem \ref{tiger} and we omit the analysis.  Also, we note that the conditions on $\mu$, $\nu$ and $n$ ensure that the integrals in all inequalities exist and are positive, and will also allow us to use inequality (\ref{Imon}) when needed.  As in the proof of Theorem \ref{tiger}, we do not comment on this further.

(i) Applying inequality (\ref{Imon}) gives
\begin{align*}\int_0^x \frac{ \tilde{t}_{\mu,\nu}(u)}{u^\nu}\,\mathrm{d}u>\int_0^x \frac{ \tilde{t}_{\mu+1,\nu+1}(u)}{u^\nu}\,\mathrm{d}u =\frac{ \tilde{t}_{\mu,\nu}(x)}{x^\nu}-\frac{x^{\mu-\nu+1}}{2^{\mu+1}\Gamma\big(\frac{\mu-\nu+3}{2}\big)\Gamma\big(\frac{\mu+\nu+3}{2}\big)},
\end{align*}
where we evaluated the integral using the differentiation formula (\ref{diffone}) and the limiting form (\ref{Itend0}).

(ii) The assertion that there is equality in (\ref{bi2}) and (\ref{bi3}) when $2\nu+n=-1$ follows because both these upper and lower bounds (which we now prove) are then equal.  Suppose now that $2\nu+n>-1$, and consider the function
\begin{equation*}w(x)=\int_0^x \frac{ \tilde{t}_{\mu+n,\nu+n}(u)}{u^\nu}\,\mathrm{d}u-\frac{ \tilde{t}_{\mu+n+1,\nu+n+1}(x)}{x^\nu}+b_{\mu,\nu,n}x^{\mu-\nu+n+2}.
\end{equation*}
We prove that $w(x)>0$ for all $x>0$, which will give the result.  Let us first note that using the differentiation formula (\ref{diffone}) followed by the relation (\ref{Iidentity}) gives that
\begin{align}&\frac{\mathrm{d}}{\mathrm{d}x}\bigg(\frac{ \tilde{t}_{\mu+n+1,\nu+n+1}(x)}{x^\nu}\bigg)=\frac{\mathrm{d}}{\mathrm{d}x}\bigg(x^{n+1}\cdot\frac{ \tilde{t}_{\mu+n+1,\nu+n+1}(x)}{x^{\nu+n+1}}\bigg)\nonumber\\
&\quad=(n+1)\frac{ \tilde{t}_{\mu+n+1,\nu+n+1}(x)}{x^{\nu+1}}+\frac{ \tilde{t}_{\mu+n+2,\nu+n+2}(x)}{x^{\nu}}+\frac{a_{\mu+n+1,\nu+n+1}(x)}{x^\nu}\nonumber
\end{align}
\begin{align}
&\quad=\frac{n+1}{2(\nu+n+1)}\bigg(\frac{ \tilde{t}_{\mu+n,\nu+n}(x)}{x^\nu}-\frac{ \tilde{t}_{\mu+n+2,\nu+n+2}(x)}{x^\nu}-\frac{a_{\mu+n+1,\nu+n+1}(x)}{x^\nu}\bigg)\nonumber\\
&\quad\quad+\frac{ \tilde{t}_{\mu+n+2,\nu+n+2}(x)}{x^\nu}+\frac{a_{\mu+n+1,\nu+n+1}(x)}{x^\nu}\nonumber \\
&\quad=\frac{n+1}{2(\nu+n+1)}\frac{ \tilde{t}_{\mu+n,\nu+n}(x)}{x^\nu}+\frac{2\nu+n+1}{2(\nu+n+1)}\frac{ \tilde{t}_{\mu+n+2,\nu+n+2}(x)}{x^\nu}\nonumber \\
\label{num3}&\quad+(\mu-\nu+n+2)b_{\mu,\nu,n}x^{\mu-\nu+n+1}.
\end{align}
Therefore
\begin{align*}w'(x)=\frac{2\nu+n+1}{2(\nu+n+1)}\bigg(\frac{ \tilde{t}_{\mu+n,\nu+n}(x)}{x^\nu}-\frac{ \tilde{t}_{\mu+n+2,\nu+n+2}(x)}{x^\nu}\bigg) >0,
\end{align*}
where we applied (\ref{Imon}) to obtain the inequality.  Also, from (\ref{Itend0}) we have, as $x\downarrow0$,
\begin{align*}w(x)&\sim \int_0^x \frac{u^{n+1}}{\sqrt{\pi}2^{\nu+n}\Gamma(\nu+n+\frac{3}{2})}\,\mathrm{d}u-\frac{x^{n+2}}{\sqrt{\pi}2^{\nu+n+1}\Gamma(\nu+n+\frac{5}{2})}+b_{\mu,\nu,n}x^{\mu-\nu+n+2}\\
&=\frac{x^{n+2}}{\sqrt{\pi}2^{\nu+n}(n+2)\Gamma(n+\nu+\frac{3}{2})} -\frac{x^{n+2}}{\sqrt{\pi}2^{\nu+n+1}\Gamma(\nu+n+\frac{5}{2})}+b_{\mu,\nu,n}x^{\mu-\nu+n+2}\\ 
&=\frac{x^{n+2}}{\sqrt{\pi}2^{\nu+n}\Gamma(\nu+n+\frac{3}{2})}\bigg(\frac{1}{n+2}-\frac{1}{2(\nu+n+\frac{3}{2})}\bigg)+b_{\mu,\nu,n}x^{\mu-\nu+n+2} >0,
\end{align*}
where the inequality holds because $\nu>-\frac{1}{2}(n+1)$.  Putting this together, we conclude that $w(x)>0$ for all $x>0$, as we required.

(iii) On integrating both sides of (\ref{num3}) over $(0,x)$, applying the fundamental theorem of calculus and rearranging we obtain
\begin{align*}\int_0^x \frac{ \tilde{t}_{\mu+n,\nu+n} (u)}{u^\nu}\,\mathrm{d}u &= \frac{2(\nu+n+1)}{n+1} \frac{ \tilde{t}_{\mu+n+1,\nu+n+1} (x)}{x^\nu} -\frac{2\nu+n+1}{n+1} \int_0^x \frac{ \tilde{t}_{\mu+n+2,\nu+n+2} (u)}{u^\nu}\,\mathrm{d}u \\
&\quad-\frac{2\nu+n+1}{n+1}\int_0^x (\mu-\nu+n+2)b_{\mu,\nu,n} u^{\mu-\nu+n+1}\,\mathrm{d}u.
\end{align*}
Inequality (\ref{bi2}) now follows on evaluating the second integral on the right hand-side of the above expression and using inequality (\ref{bi2}) to bound the first integral.

(iv) Integration by parts and inequality (\ref{bi1}) gives that
\begin{align*} \int_0^x \mathrm{e}^{-\beta u} \frac{ \tilde{t}_{\mu,\nu}(u)}{u^\nu} \,\mathrm{d}u &= \mathrm{e}^{-\beta x}\int_0^x \frac{ \tilde{t}_{\mu,\nu}(u)}{u^\nu}\,\mathrm{d}u + \beta \int_0^x \mathrm{e}^{-\beta u}\bigg(\int_0^u  \frac{ \tilde{t}_{\mu,\nu}(y)}{y^\nu} \,\mathrm{d}y\bigg) \,\mathrm{d}u \\
&> \mathrm{e}^{-\beta x}\int_0^x \frac{ \tilde{t}_{\mu,\nu}(u)}{u^\nu}\,\mathrm{d}u + \beta \int_0^x \mathrm{e}^{-\beta u} \frac{ \tilde{t}_{\mu,\nu}(u)}{u^\nu}  \,\mathrm{d}u\\
&\quad-\beta \int_0^x \frac{u^{\mu-\nu+1}\mathrm{e}^{-\beta u}}{2^{\mu+1}\Gamma\big(\frac{\mu-\nu+3}{2}\big)\Gamma\big(\frac{\mu+\nu+3}{2}\big)} \,\mathrm{d}u,
\end{align*}
whence on rearranging and recognising the final integral as a lower incomplete gamma function we obtain inequality (\ref{bi3}).

(v) Combine parts (i) and (iv).
\end{proof}

We end with an example of how one can combine the inequalities of Theorems \ref{tiger} and \ref{tiger1} and the integral formula (\ref{besint6}) to obtain lower and upper bounds for a generalized hypergeometric function.  


\begin{corollary}\label{struvebessel}Let $\mu>-\frac{3}{2}$, $-\frac{1}{2}<\nu<\mu+1$. Then, for all $x>0$,
\begin{align*} \tilde{t}_{\mu+1,\nu+1}(x)&<\frac{x^{\mu+2}}{2^{\mu+1}(\mu+\nu+2)\Gamma\big(\frac{\nu-\nu+3}{2}\big)\Gamma\big(\frac{\mu+\nu+3}{2}\big)}\\
&\quad\times{}_2F_3\bigg(1,\frac{\mu+\nu+2}{2};\frac{\mu-\nu+3}{2},\frac{\mu+\nu+3}{2},\frac{\mu+\nu+4}{2};\frac{x^2}{4}\bigg) \nonumber \\
&< \tilde{t}_{\mu+1,\nu+1}(x)\bigg\{1+\frac{1}{2\nu+1}\bigg(1-\frac{ \tilde{t}_{\mu+3,\nu+3}(x)}{ \tilde{t}_{\mu+1,\nu+1}(x)}\bigg)\bigg\}-a_{\mu,\nu,0}x^{\mu+2}.
\end{align*}
\end{corollary}

\begin{proof}Apply inequalities (\ref{besi11}) and (\ref{besi22}) (with $\beta=n=0$) of Theorem \ref{tiger} to the integral formula (\ref{besint6}) (with $\alpha=\nu$). 
\end{proof}

\begin{remark}We know from Theorem \ref{tiger} that the double inequality in Corollary \ref{struvebessel} is tight in the limit $\nu\rightarrow\infty$, and that the the upper bound is tight at $x\downarrow0$.  The double inequality is also clearly tight as $\nu\rightarrow\infty$. 

To gain further insight into the approximation, we obtained some numerical results using \emph{Mathematica}.  Let $L_{\mu,\nu}(x)$ and $U_{\mu,\nu}(x)$ denote the lower and upper bounds in the double inequality and let $F_{\mu,\nu}(x)$ denote the expression involving the generalized hypergeometric function which is bounded by these quantities.  We considered three cases of $\mu-\nu=k$, $k=-0.5,2,5$, and in each case took $\nu=0,1,2.5,5,10$.  (Tables for the case $\mu=\nu$ can be found in \cite{gaunt ineq4}.) The relative error in approximating $F_{\mu,\nu}(x)$ by $L_{\mu,\nu}(x)$ and $U_{\mu,\nu}(x)$ are reported in Tables \ref{table1} and \ref{table2}.  For a given $x$ and $\mu$, we observe that the relative error in approximating $F_{\mu,\nu}(x)$ by either $L_{\mu,\nu}(x)$ or $U_{\mu,\nu}(x)$ decreases as $\nu$ increases.  We also notice that, for a given $\mu$ and $\nu$, the relative error in approximating $F_{\mu,\nu}(x)$ by $L_{\mu,\nu}(x)$ decreases as $x$ increases.  Although, from Table \ref{table2} we see that, for a given $\mu$ and $\nu$, as $x$ increases the relative error in approximating $F_{\mu,\nu}(x)$ by $U_{\mu,\nu}(x)$ initially increases before decreasing.  Finally, we observe that the bounds are most accurate in the case $\mu-\nu=-0.5$.

\begin{table}[h]
\centering
\caption{\footnotesize{Relative error in approximating $F_{\mu,\nu}(x)$ by $L_{\mu,\nu}(x)$.}}
\label{table1}
{\scriptsize
\begin{tabular}{|c|rrrrrrr|}
\hline
 \backslashbox{$(\mu,\nu)$}{$x$}      &    0.5 &    5 &    10 &    15 &    25 &    50 & 100   \\
 \hline
$(-0.75-0.25)$ & 0.4957 & 0.2542 & 0.1118 & 0.0709 & 0.0414 & 0.0203 & 0.0101 \\
$(-0.5,0)$ & 0.3970 & 0.2221 & 0.1074 & 0.0695 & 0.0409 & 0.0202 & 0.0101  \\
$(2,2.5)$ & 0.1330 & 0.1104 & 0.0775 & 0.0570 & 0.0366 & 0.0192 & 0.0098  \\
$(4.5,5)$ & 0.0799 & 0.0732 & 0.0594 & 0.0475 & 0.0329 & 0.0182 & 0.0095 \\ 
$(9.5,10)$ & 0.0444 & 0.0430 & 0.0392 & 0.0346 & 0.0268 & 0.0164 & 0.0091 \\  
  \hline
$(1.75-0.25)$ & 0.2217 & 0.1791 & 0.1088 & 0.0709 & 0.0414 & 0.0203 & 0.0101  \\
$(2,0)$ & 0.1996 & 0.1654 & 0.1047 & 0.0694 & 0.0409 & 0.0202 & 0.0101  \\
$(4.5,2.5)$ & 0.0999 & 0.0931 & 0.0749 & 0.0568 & 0.0366 & 0.0192 & 0.0098  \\
$(7,5)$ & 0.0666 & 0.0643 & 0.0570 & 0.0472 & 0.0329 & 0.0182 & 0.0095 \\ 
$(12,10)$ & 0.0400 & 0.0394 & 0.0375  & 0.0342 & 0.0268  & 0.0164 & 0.0091 \\  
  \hline
$(4.75-0.25)$ & 0.1332 & 0.1242 & 0.0981 & 0.0712 & 0.0414 & 0.0203 & 0.0101  \\
$(5,0)$ & 0.1249 & 0.1172 & 0.0944 & 0.0686 & 0.0409 & 0.0202 & 0.0101  \\
$(7.5,2.5)$ & 0.0769 & 0.0747 & 0.0673 & 0.0555 & 0.0366 & 0.0192 & 0.0098  \\
$(10,5)$ & 0.0555 & 0.0547 & 0.0515 & 0.0457 & 0.0329 & 0.0182 & 0.0095 \\ 
$(15,10)$ & 0.0357 & 0.0355 & 0.0346  & 0.0328 & 0.0268  & 0.0164 & 0.0091 \\  
  \hline
\end{tabular}}
\end{table}
\begin{table}[h]
\centering
\caption{\footnotesize{Relative error in approximating $F_{\mu,\nu}(x)$ by $U_{\mu,\nu}(x)$.}}
\label{table2}
{\scriptsize
\begin{tabular}{|c|rrrrrrr|}
\hline
 \backslashbox{$(\mu,\nu)$}{$x$}      &    0.5 &    5 &    10 &    15 &    25 &    50 & 100   \\
 \hline
$(-0.75-0.25)$ & 0.0103 & 0.4528 & 0.4312 & 0.3268 & 0.2137 & 0.1134 & 0.0584  \\
$(-0.5,0)$ & 0.0044 & 0.1928 & 0.1967 & 0.1543 & 0.1034 & 0.0558 & 0.0290  \\
$(2,2.5)$ & 0.0001 & 0.0080 & 0.0143 & 0.0148 & 0.0125 & 0.0080 & 0.0045  \\
$(4.5,5)$ & 0.0000 & 0.0016 & 0.0038 & 0.0049 & 0.0050 & 0.0037 & 0.0023 \\ 
$(9.5,10)$ & 0.0000 & 0.0002 & 0.0007  & 0.0011 & 0.0015  & 0.0014 & 0.0010 \\   
  \hline
$(1.75-0.25)$ & 0.0038 & 0.2661 & 0.4007 & 0.3256 & 0.2137 & 0.1134 & 0.0584  \\
$(2,0)$ & 0.0016 & 0.1136 & 0.1823 & 0.1536 & 0.1034 & 0.0558 & 0.0290  \\
$(4.5,2.5)$ & 0.0001 & 0.0052 & 0.0129 & 0.0147 & 0.0125 & 0.0080 & 0.0045  \\
$(7,5)$ & 0.0000 & 0.0011 & 0.0034 & 0.0048 & 0.0050 & 0.0037 & 0.0023 \\ 
$(12,10)$ & 0.0000 & 0.0002 & 0.0006  & 0.0010 & 0.0015  & 0.0014 & 0.0010 \\ 
\hline
$(4.75-0.25)$ & 0.0014 & 0.1217 & 0.2999 & 0.3120 & 0.2137 & 0.1134 & 0.0584  \\
$(5,0)$ & 0.0006 & 0.0534 & 0.1361 & 0.1468 & 0.1034 & 0.0558 & 0.0290  \\
$(7.5,2.5)$ & 0.0000 & 0.0030 & 0.0095 & 0.0136 & 0.0125 & 0.0080 & 0.0045  \\
$(10,5)$ & 0.0000 & 0.0007 & 0.0026 & 0.0043 & 0.0050 & 0.0037 & 0.0023 \\ 
$(15,10)$ & 0.0000 & 0.0001 & 0.0005  & 0.0009 & 0.0014  & 0.0014 & 0.0010 \\  
  \hline
\end{tabular}}
\end{table}

\end{remark}

\subsection*{Acknowledgements}
The author is supported by a Dame Kathleen Ollerenshaw Research Fellowship.

\end{document}